\documentclass[11pt]{amsart}
\usepackage{color}
\usepackage{amssymb,amsthm,amsmath,epsfig,latexsym}
\usepackage{calc,times,verbatim}
\usepackage{geometry} 
\begin{document}

\newcommand{\mmbox}[1]{\mbox{${#1}$}}
\newcommand{\proj}[1]{\mmbox{{\mathbb P}^{#1}}}
\newcommand{\Cr}{C^r(\Delta)}
\newcommand{\CR}{C^r(\hat\Delta)}
\newcommand{\affine}[1]{\mmbox{{\mathbb A}^{#1}}}
\newcommand{\Ann}[1]{\mmbox{{\rm Ann}({#1})}}
\newcommand{\caps}[3]{\mmbox{{#1}_{#2} \cap \ldots \cap {#1}_{#3}}}
\newcommand{\Proj}{{\mathbb P}}
\newcommand{\N}{{\mathbb N}}
\newcommand{\Z}{{\mathbb Z}}
\newcommand{\R}{{\mathbb R}}
\newcommand{\A}{{\mathcal{A}}}
\newcommand{\Tor}{\mathop{\rm Tor}\nolimits}
\newcommand{\Ext}{\mathop{\rm Ext}\nolimits}
\newcommand{\Hom}{\mathop{\rm Hom}\nolimits}
\newcommand{\im}{\mathop{\rm Im}\nolimits}
\newcommand{\rank}{\mathop{\rm rank}\nolimits}
\newcommand{\supp}{\mathop{\rm supp}\nolimits}
\newcommand{\arrow}[1]{\stackrel{#1}{\longrightarrow}}
\newcommand{\CB}{Cayley-Bacharach}
\newcommand{\coker}{\mathop{\rm coker}\nolimits}
\sloppy
\theoremstyle{plain}

\newtheorem*{thm*}{Theorem}
\newtheorem{defn0}{Definition}[section]
\newtheorem{prop0}[defn0]{Proposition}
\newtheorem{quest0}[defn0]{Question}
\newtheorem{thm0}[defn0]{Theorem}
\newtheorem{lem0}[defn0]{Lemma}
\newtheorem{corollary0}[defn0]{Corollary}
\newtheorem{example0}[defn0]{Example}
\newtheorem{remark0}[defn0]{Remark}
\newtheorem{conj0}[defn0]{Conjecture}

\newenvironment{defn}{\begin{defn0}}{\end{defn0}}
\newenvironment{conj}{\begin{conj0}}{\end{conj0}}
\newenvironment{prop}{\begin{prop0}}{\end{prop0}}
\newenvironment{quest}{\begin{quest0}}{\end{quest0}}
\newenvironment{thm}{\begin{thm0}}{\end{thm0}}
\newenvironment{lem}{\begin{lem0}}{\end{lem0}}
\newenvironment{cor}{\begin{corollary0}}{\end{corollary0}}
\newenvironment{exm}{\begin{example0}\rm}{\end{example0}}
\newenvironment{rem}{\begin{remark0}\rm}{\end{remark0}}

\newcommand{\defref}[1]{Definition~\ref{#1}}
\newcommand{\conjref}[1]{Conjecture~\ref{#1}}
\newcommand{\propref}[1]{Proposition~\ref{#1}}
\newcommand{\thmref}[1]{Theorem~\ref{#1}}
\newcommand{\lemref}[1]{Lemma~\ref{#1}}
\newcommand{\corref}[1]{Corollary~\ref{#1}}
\newcommand{\exref}[1]{Example~\ref{#1}}
\newcommand{\secref}[1]{Section~\ref{#1}}
\newcommand{\remref}[1]{Remark~\ref{#1}}
\newcommand{\questref}[1]{Question~\ref{#1}}

\newcommand{\std}{Gr\"{o}bner}
\newcommand{\jq}{J_{Q}}

\title{On the Geramita-Harbourne-Migliore conjecture}
\author{\c{S}tefan O. Toh\v{a}neanu and Yu Xie}

. \keywords{linear free resolution, fold products, star configuration, symbolic power. \\
\indent Toh$\breve{{\rm a}}$neanu's address: Department of Mathematics, University of Idaho, Moscow, Idaho 83844-1103, USA, Email: tohaneanu@uidaho.edu.\\
\indent Xie's address: Department of Mathematics, Widener University, Chester, Pennsylvania 19013, USA, Email: yxie@widener.edu.}

\begin{abstract}
Let $\Sigma$ be a finite collection of linear forms in $\mathbb K[x_0,\ldots,x_n]$, where $\mathbb K$ is a field. Denote ${\rm Supp}(\Sigma)$ to be the set of all nonproportional elements of $\Sigma$, and suppose ${\rm Supp}(\Sigma)$ is generic, meaning that any $n+1$ of its elements are linearly independent. Let $1\leq a\leq |\Sigma|$. In this article we prove the conjecture that the ideal generated by (all) $a$-fold products of linear forms of $\Sigma$ has linear graded free resolution. As a consequence we prove the Geramita-Harbourne-Migliore conjecture concerning the primary decomposition of ordinary powers of defining ideals of star configurations, and we also determine the resurgence of these ideals.
\end{abstract}

\maketitle

\section{Introduction}

Let $R:=\mathbb K[x_0,\ldots,x_n]$ be the ring of polynomials with coefficients in an arbitrary field $\mathbb K$, thought of as a graded ring with the standard grading. Denote $M:=\langle x_0,\ldots,x_n\rangle$ the irrelevant ideal in $R$.

Let $\Sigma$ be a finite collection of linear forms $L_1,\ldots,L_N$ in $R$, some possibly proportional. Denote $\Sigma=(L_1,\ldots, L_N)$. For any $1\leq a\leq N$, define
$$I_a(\Sigma):=\langle\{L_{i_1}\cdots L_{i_a}|1\leq i_1<\cdots<i_a\leq N\}\rangle,
$$
the {\em ideal generated by (all) $a$-fold products of linear forms of \, $\Sigma$}. For convention, $I_0(\Sigma)=R$, and $I_a(\Sigma)=0$, for any $a\geq N+1$. Often we will denote $I_a(\Sigma)$ with $I_a(L_1\cdots L_N)$.

The {\em rank} of $\Sigma$, denoted ${\rm rk}(\Sigma)$, is ${\rm ht}(\langle L_1,\ldots,L_N\rangle)$.

\subsection{Linear Codes.} The ideals generated by $a$-fold products of linear forms originally occurred in coding theory context as a nice tool to determine the minimum distance of a linear code (see \cite{DePe}).

Let $L=c_0x_0+\cdots+c_nx_n$ be an arbitrary element of $\Sigma$. Dually we get a column vector $(c_0,\ldots,c_n)^T\in\mathbb K^{n+1}$. This way we build an $(n+1)\times N$ matrix $G_{\Sigma}$, and consequently a linear code $\mathcal C_{\Sigma}$ which is the image of the linear map $\displaystyle \mathbb K^{n+1}\stackrel{\cdot G_{\Sigma}}\longrightarrow\mathbb K^N$. This is a linear code of block-length $N$ and dimension ${\rm rk}(G_{\Sigma})={\rm rk}(\Sigma)$; a generating matrix of $\mathcal C_{\Sigma}$ is $G_{\Sigma}$. Conversely, to any linear code we can associate a collection of linear forms dual to the columns of some generating matrix.

Suppose ${\rm rk}(\Sigma)=n+1$. For any $1\leq r\leq n+1$ one can define the $r$-th generalized Hamming weight, $d_r(\mathcal C_{\Sigma})$, which by classical results in coding theory (see for example \cite[Corollary 1.3 and Proposition 1.7]{AnGaTo}) has the following description: $N-d_r(\mathcal C_{\Sigma})$ is the maximum number of columns of $G_{\Sigma}$ that span a $n+1-r$ dimensional vector space. Observe $d_1(\mathcal C_{\Sigma})$ is the usual minimum distance of $\mathcal C_{\Sigma}$. Moreover,  $N-d_{n+1}(\mathcal C_{\Sigma})=0$, since $G_{\Sigma}$ has no zero columns, and $N-d_{n}(\mathcal C_{\Sigma})$ is the maximum number of columns of $G_{\Sigma}$ that are proportional to each-other (i.e., the maximum number of linear forms of $\Sigma$ that are proportional to each-other).

The generalized Hamming weights help determine the heights of ideals generated by $a$-fold products of linear forms. By \cite[Proposition 2.2]{AnGaTo}, for $r=1,\ldots,n+1$, with the convention that $d_0(\mathcal C_{\Sigma})=0$, and for any $d_{r-1}(\mathcal C_{\Sigma})<a\leq d_{r}(\mathcal C_{\Sigma})$, one has
$${\rm ht}(I_a(\Sigma))=n+2-r.$$ For example, if $1\leq a\leq d_1(\mathcal C_{\Sigma})$, then ${\rm ht}(I_a(\Sigma))=n+1$, the maximum possible value (see also \cite{DePe}).

The connections between homological properties of ideals generated by $a$-fold products of linear forms and linear codes and their combinatorics is further transparent in \cite[Theorem 2.8]{AnGaTo}, that presents a formula for the degree of $I_a(\Sigma)$, or in \cite[Proposition 2.10]{AnGaTo}, that presents a formula for the minimum number of generators of $I_a(\Sigma)$, both in terms of the coefficients of the Tutte polynomial of the matroid of $G_{\Sigma}$.

\medskip

Let $I\subset R$ be a homogeneous ideal generated in degree $d$. We say that the $R$-module $I$, or $R/I$, has {\em linear (minimal) graded free resolution}, if one has the graded free resolution $$0\rightarrow R^{n_{b+1}}(-(d+b))\rightarrow\cdots\rightarrow R^{n_2}(-(d+1))\rightarrow R^{n_1}(-d)\rightarrow R \rightarrow R/I\rightarrow 0,$$ for some positive integer $b$.

The following is a conjecture regarding to the minimal free resolution of ideals generated by $a$-fold products of linear forms.

\smallskip

\begin{conj}(See for example \cite[Conjecture 1]{AnGaTo})  \label{Conjecture2} \, For any $\Sigma=(L_1,\ldots, L_N)$ and any $1\leq a\leq N$, the ideal $I_a(\Sigma)$ has linear graded free resolution.
\end{conj}

\smallskip

This conjecture was made six years ago and some consistent progress  has been done towards proving this conjecture. Now in this article we prove it when the support of $\Sigma$ is generic (Theorem \ref{thm_linear}).

\smallskip

\begin{itemize}
  \item[(a)] First evidence of the validity of this conjecture was observed in \cite[Theorem 3.1]{To3}, where it is shown that for any $1\leq a\leq d_1(\mathcal C_{\Sigma})$, one has $I_a(\Sigma)=M^a$, and powers of the irrelevant ideal (more generally, of any linear prime) have linear graded free resolution (see for example \cite[Corollary 1.5]{EiGo}).
  \item[(b)] If $n=1$, so $\Sigma\subset\mathbb K[x_0,x_1]$, \cite[Theorem 2.2]{To1} proves that for any $1\leq a\leq N$, $I_a(\Sigma)$ has linear graded free resolution.
  \item[(c)] If $a=N$, then $I_N(\Sigma)=\langle L_1\cdots L_N\rangle$, which has linear graded free resolution as it is a principal ideal.
  \item[(d)] If $a=N-1$, then \cite[Section 2.1]{To1} shows that $I_{N-1}(\Sigma)$ has linear graded free resolution.
  \item[(e)] If $a=N-2$, then \cite[Theorem 2.5]{To1} shows that $I_{N-2}(\Sigma)$ has linear graded free resolution.
  \item[(f)] Suppose $\Sigma$ has no proportional linear forms, i.e., $\Sigma$ defines a hyperplane arrangement in $\mathbb P^n$. If $\Sigma$ is generic (meaning any ${\rm rk}(\Sigma)$ linear forms of $\Sigma$ are linearly independent), then for any $1\leq a\leq N$, $I_a(\Sigma)$ has linear graded free resolution. Indeed, after a change of variables one can assume that ${\rm rk}(\Sigma)=n+1$. Then $d_1(\mathcal C_{\Sigma})=N-(n+1)+1=N-n$, and part (a) above shows the result for $1\leq a\leq N-n$. For $N-n+1\leq a\leq N$, via Lemma \ref{lemma1} below, $I_a(\Sigma)$ is the defining ideal of a star configuration, and these have been shown to have linear graded free resolution (given by the Eagon-Northcott complex): see \cite[Remark 2.11]{GeHaMi}, \cite[Example 3.4 and Corollary 3.5]{GeHaMiNa}, or \cite[Corollary 3.7]{PaSh}. Independently, one can also obtain this result by applying the proof of \cite[Theorem 2.5]{GaTo} and \cite[Proposition 3.5]{AnGaTo}.

  \item[(f')] More generally, whenever $R/I_a(\Sigma)$ is Cohen-Macaulay, then $I_a(\Sigma)$ has a linear graded free resolution. This can be seen from the discussions at the end of the proof of \cite[Proposition 2.1]{To2}.
  \item[(g)] Generalizing the main result in \cite{AnTo}, in \cite{To4} it is shown that for $a=d_1(\mathcal C_{\Sigma})+1$, $I_a(\Sigma)$ has linear graded free resolution, for any $\Sigma$, a collection of linear forms.
  \item[(h)] If $\Sigma$ defines a line arrangement in $\mathbb P^2$, then by \cite{To4}, for any $1\leq a\leq N$, $I_a(\Sigma)$ has linear graded free resolution.
\end{itemize}

\medskip

In our first main result, Theorem \ref{thm_linear}, we add one more item to this list. Suppose $\Sigma=(\underbrace{\ell_1,\ldots,\ell_1}_{m_1},\ldots,\underbrace{\ell_s,\ldots,\ell_s}_{m_s})$ with $\gcd(\ell_i,\ell_j)=1$ if $i\neq j$. The {\em support of $\Sigma$} is ${\rm Supp}(\Sigma):=\{\ell_1,\ldots,\ell_s\}$. If the support of $\Sigma$ is generic, i.e., any ${\rm rk}(\Sigma)$ elements of ${\rm Supp}(\Sigma)$ are linearly independent, then we show that $I_a(\Sigma)$ has linear graded free resolution, for any $1\leq a\leq m_1+\cdots+m_s$.


\subsection{Star Configurations.} Let $\A=\{H_1,\ldots,H_s\}$ be a collection of $s\geq n+1$ hyperplanes in $\mathbb P^n$, and suppose $\ell_1,\ldots,\ell_s\in R$ are defining linear forms of these hyperplanes: i.e., $H_i=V(\ell_i), i=1,\ldots,s$. Suppose $\A$ is generic, meaning that any $n+1$ of the defining linear forms are linearly independent, or in the language of \cite[Definition 2.1]{GeHaMi}, the hyperplanes $H_1,\ldots,H_s$ {\em meet properly}.

Let $1\leq c\leq n$ be an integer and define {\em the star configuration of codimension $c$ with support $\A$ (in $\mathbb P^n$)} to be
$$V_c(\A):=\bigcup_{1\leq j_1<\cdots<j_c\leq s}H_{j_1}\cap\cdots\cap H_{j_c}.$$ It is clear that the defining ideal is
$$I(V_c(\A))=\bigcap_{1\leq j_1<\cdots<j_c\leq s}\langle \ell_{j_1},\ldots,\ell_{j_c}\rangle,$$ and the {\em $m$-th symbolic power} of this ideal is $$I(V_c(\A))^{(m)}=\bigcap_{1\leq j_1<\cdots<j_c\leq s}\langle \ell_{j_1},\ldots,\ell_{j_c}\rangle^m.$$

The following is a conjecture about the ordinary powers of $I(V_c(\A))$.

\begin{conj}(See \cite[Conjecture 4.1]{GeHaMi})  \label{Conjecture1}  \, For any $m\geq 1$ one has
$$I(V_c(\A))^m=I(V_c(\A))^{(m)}\cap I(V_{c+1}(\A))^{(2m)}\cap\cdots\cap I(V_n(\A))^{((n-c+1)m)}\cap M^{(s-c+1)m}.$$
\end{conj}

\begin{rem}\label{remark0} Let $J\subset R$ be an ideal generated in degree $d$. Then $J\subseteq J^{\rm sat}\cap M^d$.\footnote{If $I\subset R$ is a homogeneous ideal, by definition $I^{\rm sat}:=\{f\in R| \exists \, n(f)\geq 0 \mbox{ such that }M^{n(f)}\cdot f\subset I\}$.} If $R/J$ has linear graded free resolution (equivalently, ${\rm reg}(R/J)=d-1$), since ${\rm H}_{\frak m}^0(R/J)=J^{\rm sat}/J$, by \cite[Theorem 4.3]{Ei}, we have $(J^{\rm sat}/J)_e=0, \mbox{ for any }e\geq d$. This means that $J^{\rm sat}\cap M^d\subseteq J$, and therefore $$J=J^{\rm sat}\cap M^d.$$

It is shown in \cite[Corrolary 4.9]{GeHaMi} that for any $m\geq 1$,
$$\left(I(V_c(\A))^m\right)^{\rm sat}=I(V_c(\A))^{(m)}\cap I(V_{c+1}(\A))^{(2m)}\cap\cdots\cap I(V_n(\A))^{((n-c+1)m)}.$$
 Also by \cite[Proposition 2.9 (4)]{GeHaMi}, one has that $I(V_c(\A))$ is generated in degree $s-c+1$, and therefore, $I(V_c(\A))^m$ is generated in degree $(s-c+1)m$.
So in order to prove Conjecture \ref{Conjecture1}, it is enough to show that $I(V_c(\A))^m$ has linear graded free resolution.
\end{rem}

\smallskip

 Conjecture \ref{Conjecture1}  has been verified to be true in the following cases (See \cite[Remark 4.4]{GeHaMi}):
\begin{itemize}
  \item[(i)] $m=1$. This is true because of Remark \ref{remark0} and item (f) above.
  \item[(ii)] $c=1$. This is true because $I(V_c(\A))$ is a principal ideal.
  \item[(iii)] $c=n$. This is true by \cite[Lemma 2.3.3(c), Lemma 2.4.2]{BoHa}.
  \item[(iv)] $n=s-1$. This is true by \cite[Theorem 4.8]{GeHaMi}.
\end{itemize}

We are not aware at this moment of any other consistent progress in proving this conjecture, other than the observations we make in Remark \ref{remark2}. Nonetheless, by identifying ordinary powers of defining ideals of star configurations with ideals generated by $a$-fold products of linear forms (see Lemma \ref{lemma1}), then Theorem \ref{thm_linear}, Remark \ref{remark0}, and \cite[Corrolary 4.9]{GeHaMi} will prove Conjecture \ref{Conjecture1} in its full generality (see Theorem \ref{thm_GHMconj}).

In the end, again by using the identification provided by Lemma \ref{lemma1}, we give a positive answer to \cite[Question 4.12]{GeHaMi} and prove a result that calculates the resurgence of the defining ideal of any star configuration (see Theorem \ref{thm_resurgence}), which generalizes \cite[Theorem 4.11]{GeHaMi}.

\section{Ideals generated by $a$-fold products of linear forms have linear graded free resolution}

Let $\Sigma:= (\underbrace{\ell_1,\ldots,\ell_1}_{m_1},\ldots,\underbrace{\ell_s,\ldots,\ell_s}_{m_s})$ be a collection of $N:=m_1+\cdots+m_s$ linear forms in $R:=\mathbb K[x_0,\ldots,x_n]$, with $s, m_1,\ldots,m_s\geq 1$, and $\gcd(\ell_i,\ell_j)=1$ if $i\neq j$. For any $a\geq 0$, consider the ideal generated by $a$-fold products of linear forms of $\Sigma$ $$I_a(\Sigma)=I_a(\ell_1^{m_1}\cdots\ell_s^{m_s}),$$ with the convention that if $a=0$ this ideal equals the entire ring $R$, and if $a>N$ this ideal equals the zero ideal.

One of the most delicate issue is to find a primary decomposition for $I_a(\ell_1^{m_1}\cdots\ell_s^{m_s})$ (and therefore generalizing \cite[Proposition 2.3]{AnGaTo}). We have a first lemma.

\begin{lem}\label{decomposition} For any $1\leq a\leq N$, one has
$$
I_a(\ell_1^{m_1}\cdots\ell_s^{m_s})\subseteq \bigcap_{c=1}^s\left(\bigcap_{1\leq i_1<\cdots<i_c\leq s}\langle\ell_{i_1},\ldots,\ell_{i_c}\rangle^{\mu(i_1,\ldots,i_c)}\right),
$$
where $\displaystyle \mu(i_1,\ldots,i_c):=a-\sum_{j\in [s]\setminus\{i_1,\ldots,i_c\}}m_j$ and any ideal with power $\leq 0$ is replaced by $R$.
\end{lem}

\begin{proof}  Let $\xi=\ell_1^{t_1}\cdots \ell_s^{t_s}$  be a  minimal ``monomial'' generator in $I_a(\ell_1^{m_1}\cdots\ell_s^{m_s})$, so $\sum_{j=1}^st_j=a$ and $t_j\leq m_j$ for $1\leq j\leq s$. We will show that $\xi\in \langle\ell_{i_1},\ldots,\ell_{i_c}\rangle^{\mu(i_1,\ldots,i_c)}$ for every $1\leq i_1<\cdots<i_c\leq s$ and $1\leq c\leq s$.

To do that we just need to show that
$\sum_{j\in \{i_1,\ldots,i_c\}} t_{j}\geq \mu(i_1,\ldots,i_c)=a-\sum_{j\in [s]\setminus\{i_1,\ldots,i_c\}}m_j$. This follows by the fact that if it does not hold, then $\sum_{j\in [s]\setminus\{i_1,\ldots,i_c\}} t_{j}> \sum_{j\in [s]\setminus\{i_1,\ldots,i_c\}}m_j$, which forces $t_j>m_j$ for some $j\in [s]\setminus\{i_1,\ldots,i_c\}$, a contradiction.
\end{proof}

\medskip

Next, we review the beginning of Section 2 in \cite{To3}. Let $\Sigma=(L_1,\ldots,L_N)$ be a collection of linear forms in $R$. Let $1\leq a\leq N$ and consider $I:=I_a(L_1\cdots L_N)$. Let $\frak p$ be a prime ideal containing $I$, and let $L_{i_1}\cdots L_{i_a}, 1\leq i_1<\cdots<i_a\leq N$, be an arbitrary generator of $I$. For convenience, suppose $i_1=1,\ldots,i_a=a$. Then one of $L_1,\ldots,L_a$ belongs to $\frak p$; say $L_1\in\frak p$. But $L_2\cdots L_{a+1}\in I$, and therefore one of $L_2,\ldots,L_{a+1}$ belongs to $\frak p$; say $L_2\in\frak p$ (here it may happen that $L_1$ and $L_2$ are proportional; it doesn't matter to the argument). So any prime ideal containing $I$ must contain at least $N-a+1$ linear forms from $\Sigma$ (counted with multiplicity), and conversely, any linear prime generated by $N-a+1$ linear forms of $\Sigma$ will contain $I$.

For a prime ideal $\frak p$, consider {\em the closure of $\frak p$ in $\Sigma$}, denoted $cl_{\Sigma}(\frak p)$, to be the set of all elements of $\Sigma$, considered with multiplicity/repetition, that belong to ${\frak p}$. Denote $\nu_{\Sigma}(\frak p):=|cl_{\Sigma}(\frak p)|$. Immediately, $\frak p\supseteq I$ if and only if $\nu_{\Sigma}(\frak p)\geq N-a+1$.

\medskip

Now we go back to $\Sigma= (\underbrace{\ell_1,\ldots,\ell_1}_{m_1},\ldots,\underbrace{\ell_s,\ldots,\ell_s}_{m_s})$. The linear primes showing up in the intersection in Lemma \ref{decomposition} may not be distinct. This justifies considering $\Gamma(\Sigma)$, the set of all of the pairwise distinct primes of the form $\langle\ell_{i_1},\ldots,\ell_{i_c}\rangle$, where $1\leq i_1<\cdots<i_c\leq s$ and $1\leq c\leq s$.

For any $\mathfrak{p}\in \Gamma(\Sigma)$ one has $\frak p=\langle\ell_{i_1},\ldots,\ell_{i_c}\rangle$ for some linear forms $\ell_{i_1},\ldots,\ell_{i_c}$, where $\displaystyle \langle\ell_{i_1},\ldots,\ell_{i_c}\rangle^{\mu(i_1,\ldots,i_c)}$ is  a factor of the intersection in Lemma \ref{decomposition}.  If $\langle\ell_{j_1},\ldots,\ell_{j_{c'}}\rangle=\frak p$, then in the intersection we will ``only see'' $\displaystyle{\frak p}^{\max\{\mu(i_1,\ldots,i_c),\mu(j_1,\ldots,j_{c'})\}}$. So the maximum power of $\frak p$ that can occur is the maximum $\mu(i_1,\ldots,i_u)$, whenever $\langle \ell_{i_1},\ldots,\ell_{i_u}\rangle=\frak p$.

Suppose the closure of $\frak p$ consists of $\ell_{i_1},\ldots,\ell_{i_u}$ taken with their corresponding multiplicities. So ${\frak p}=\langle \ell_{i_1},\ldots,\ell_{i_u}\rangle$, and $\nu_{\Sigma}(\frak p)=m_{i_1}+\cdots+m_{i_u}$. So $$\mu(i_1,\ldots,i_u)=a-(N-\nu_{\Sigma}(\frak p)),$$ which is maximal possible, since we cannot have more than $\nu_{\Sigma}(\frak p)$ linear forms that generate $\frak p$. Hence we have
 $$
 I\subseteq \bigcap_{c=1}^s\left(\bigcap_{1\leq i_1<\cdots<i_c\leq s}\langle\ell_{i_1},\ldots,\ell_{i_c}\rangle^{\mu(i_1,\ldots,i_c)}\right)=\bigcap_{\frak p\in \Gamma(\Sigma)}\frak p^{a-N+\nu_{\Sigma}(\frak p)},
 $$ where of course, $a-N+\nu_{\Sigma}(\frak p)\leq 0$ if and only if $I\nsubseteq\frak p$. For this reason, from now on $\Gamma(\Sigma)$ will consist only of primes that contain $I$. It is worth mentioning that $M=\langle \ell_1,\ldots,\ell_s\rangle=\langle x_0,\ldots,x_n\rangle\in\Gamma(\Sigma)$, and $\nu_{\Sigma}(M)=N$.

\bigskip

The next result is very relevant for the situation when the support of $\Sigma$ is generic, i.e., any ${\rm rk}(\Sigma)$ linear forms from ${\rm Supp}(\Sigma):=\{\ell_1,\ldots,\ell_s\}$ are linearly independent. Via a change of variables, and an embedding into a ring with fewer variables, one can suppose ${\rm rk}(\Sigma)=n+1$ (maximum rank).

\begin{prop}\label{prop_saturation} Let $\Sigma:=(\underbrace{\ell_1,\ldots,\ell_1}_{m_1},\ldots,\underbrace{\ell_s,\ldots,\ell_s}_{m_s})$ be a collection of linear forms of rank $n+1$ in $R:=\mathbb K[x_0,\ldots,x_n]$, with $s, m_1,\ldots,m_s\geq 1$, and ${\rm Supp}(\Sigma)$ generic. Set $N:=m_1+\cdots+m_s$ and for any $1\leq a\leq N$, denote $I:=I_a(\ell_1^{m_1}\cdots\ell_s^{m_s})$. Then
$$
I^{\rm sat}=\bigcap_{\frak p\in \Gamma(\Sigma)\setminus\{M\}}\frak p^{a-N+\nu_{\Sigma}(\frak p)}.$$
\end{prop}
\begin{proof} Denote the big intersection on the righthand side with $\mathbb I$. It is clear that $\mathbb I$ is a saturated ideal. Also, $I\subseteq\mathbb I$ (from Lemma \ref{decomposition}). So it is enough to show that $IR_{P}\supseteq\mathbb I R_{P}$, for any $P\in {\rm Ass}(I)\setminus\{M\}$.

Let $\frak p:=\langle cl_{\Sigma}(P)\rangle$. Since $I\subset P$ and $P\neq M$, then $\frak p\in\Gamma(\Sigma)\setminus\{M\}$.

Suppose ${\frak p}=\langle \ell_1,\ldots,\ell_c\rangle$, where, because ${\rm Supp}(\Sigma)$ is generic, we have ${\rm ht}(\frak p)=c\leq n$. Also because ${\rm Supp}(\Sigma)$ is generic, and since $c\leq n$, we have that $\ell_{c+1},\ldots,\ell_s\notin\frak p$, and by the way $\frak p$ was constructed, these linear forms are not in $P$ as well. So they become invertible under localization.

We have\footnote{For any collection $(L_1,\ldots,L_N)$ of linear forms, and for any $1\leq a\leq N$, one has $I_a(L_1\cdots L_N)=L_N I_{a-1}(L_1\cdots L_{N-1})+I_a(L_1\cdots L_{N-1})$; then we apply this several times.} $$I=I_a(\ell_1^{m_1}\cdots\ell_{s-1}^{m_{s-1}}\ell_s^{m_s})=\ell_s^{m_s}I_{a-m_s}(\ell_1^{m_1}\cdots\ell_{s-1}^{m_{s-1}})+ \cdots+\ell_s I_{a-1}(\ell_1^{m_1}\cdots\ell_{s-1}^{m_{s-1}})+I_a(\ell_1^{m_1}\cdots\ell_{s-1}^{m_{s-1}}).$$ Under localization, $\ell_s$ is invertible, and since $$I_{a-m_s}(\ell_1^{m_1}\cdots\ell_{s-1}^{m_{s-1}})\supset\cdots\supset I_a(\ell_1^{m_1}\cdots\ell_{s-1}^{m_{s-1}}),$$ we have $$I R_{P}=I_{a-m_s}(\ell_1^{m_1}\cdots\ell_{s-1}^{m_{s-1}}) R_{P}.$$ Doing this for all $\ell_{c+1},\ldots,\ell_s$ we have

$$I R_{P}=I_{a-N+\nu_{\Sigma}(\frak p)}(\ell_1^{m_1}\cdots\ell_c^{m_c}) R_{P},$$ since $\nu_{\Sigma}(\frak p)=m_1+\cdots+m_c$, which is $\geq N-a+1$.

Also, under localization in the intersection $\mathbb I$ we can ``see'' only $\frak q\in\Gamma(\Sigma)$, with $\frak q\subseteq\frak p$. Furthermore, if there is an $\ell\in \Sigma$ such that $\ell\in\frak q$, then $\ell\in\{\ell_1,\ldots,\ell_c\}$. This is true because otherwise we will have $c+1\leq n+1$ or fewer elements of ${\rm Supp}(\Sigma)$ that are linearly dependent.

After a change of variables, suppose $\ell_1=x_1,\ldots,\ell_c=x_c$. Everything put together gives

$$I R_{P}=I_{a-N+\nu_{\Sigma}(\frak p)}(x_1^{m_1}\cdots x_c^{m_c})R_{P},$$ and

$$\mathbb I R_{P}=\bigcap_{k=1}^c\left(\bigcap_{1\leq j_1<\cdots< j_k\leq c}\langle x_{j_1},\ldots,x_{j_k}\rangle^{a-N+m_{j_1}+\cdots+m_{j_k}}\right)\,R_{P}.$$

Let $b:=a-N+\nu_{\Sigma}(\frak p)=a-N+m_1+\cdots+m_c\geq 1$. Also let $d=m_1+\cdots+m_c$. Therefore,
$$a-N+m_{j_1}+\cdots+m_{j_k}=b-d+m_{j_1}+\cdots+m_{j_k}.$$

In order to show our equality under localization it is enough to show

$$I_b(x_1^{m_1}\cdots x_c^{m_c})=\bigcap_{k=1}^c\left(\bigcap_{1\leq j_1<\cdots< j_k\leq c}\langle x_{j_1},\ldots,x_{j_k}\rangle^{\mu(j_1,\ldots,j_k)}\right),$$ where $\mu(j_1,\ldots,j_k):=b-d+m_{j_1}+\cdots+m_{j_k}$.

But in Lemma \ref{decomposition} we have seen that the inclusion $\subseteq$ is satisfied.

To prove the other inclusion, let $\xi$ be a monomial in the intersection of ideals of the right side\footnote{We can work just with monomials, since the righthand side is the intersection of monomial ideals.}.

By taking $k=c$ in the intersection, since $\xi\in \langle x_1,\ldots,x_c\rangle^b$, one can write $\xi=x_1^{t_1}\cdots x_c^{t_c}$ with $\sum_{j=1}^ct_j\geq b$.

If $t_j\geq m_j$ for all $j=1,\ldots,c$, then $\xi=(x_1^{t_1-m_1}\cdots x_c^{t_c-m_c})x_1^{m_1}\cdots x_c^{m_c}$. This is an element of $I_d(x_1^{m_1}\cdots x_c^{m_c})$, which in turn is included in $I_b(x_1^{m_1}\cdots x_c^{m_c})$, as $b\leq d$.

Let $\{x_{j_1}, \ldots, x_{j_k}\} \subseteq \{x_1, \ldots, x_c\}$ such that for any $i\in \{j_1,\ldots,j_k\}$, $t_i\leq m_i$ and for any $i\in [c]\setminus\{j_1,\ldots,j_k\}$, $t_i>  m_i$.
 Now
 $$\xi=x_1^{t_1}\cdots x_c^{t_c}=\left(\prod_{i\in \{j_1,\ldots,j_k\}}x_i^{t_i} \cdot \prod_{i\in [c]\setminus\{j_1,\ldots,j_k\}}x_i^{m_i}\right)\cdot \prod_{i\in [c]\setminus\{j_1,\ldots,j_k\}}x_i^{t_i-m_i}.$$

 Set
 $\xi^{\prime} = \prod_{i\in \{j_1,\ldots,j_k\}}x_i^{t_i} \cdot \prod_{i\in [c]\setminus\{j_1,\ldots,j_k\}}x_i^{m_i}$. Then, to prove $\xi\in I_b(x_1^{m_1}\cdots x_c^{m_c})$, we just need to show $\xi^{\prime}\in I_b(x_1^{m_1}\cdots x_c^{m_c})$.

Since the exponent of every $x_i$ in $\xi^{\prime}$ is less than or equal to $m_i$, we just need to show $\sum_{i\in \{j_1,\ldots,j_k\}}t_i+\sum_{i\in [c]\setminus\{j_1,\ldots,j_k\}}m_i\geq b$.

This follows by the fact that $\xi\in \langle x_{j_1},\ldots,x_{j_k}\rangle^{\mu(j_1,\ldots,j_k)}$, therefore $$\sum_{i\in \{j_1,\ldots,j_k\}}t_i\geq  \mu(j_1,\ldots,j_k).$$ Hence
\begin{eqnarray}
\sum_{i\in \{j_1,\ldots,j_k\}}t_i+\sum_{i\in [c]\setminus\{j_1,\ldots,j_k\}}m_i&\geq& \mu(j_1,\ldots,j_k)+\sum_{i\in [c]\setminus\{j_1,\ldots,j_k\}}m_i\nonumber\\
&=&b-\sum_{i\in [c]\setminus\{j_1,\ldots,j_k\}}m_i+\sum_{i\in [c]\setminus\{j_1,\ldots,j_k\}}m_i\nonumber\\
&=&b\nonumber
\end{eqnarray} and the proof is completed.
\end{proof}

\medskip

Now we are ready to prove Conjecture \ref{Conjecture2} when the support of $\Sigma$ is generic.

\smallskip

\begin{thm}\label{thm_linear} Let $\Sigma:=(\underbrace{\ell_1,\ldots,\ell_1}_{m_1},\ldots,\underbrace{\ell_s,\ldots,\ell_s}_{m_s})$ be a collection of linear forms in $R:=\mathbb K[x_0,\ldots,x_n]$, with $s, m_1,\ldots,m_s\geq 1$, and with ${\rm Supp}(\Sigma)$ generic. Set $N:=m_1+\cdots+m_s$. Then for any $1\leq a\leq N$, the ideal $I_a(\ell_1^{m_1}\ell_2^{m_2}\cdots\ell_s^{m_s})$ has linear graded free resolution.
\end{thm}

\begin{proof} We are proving the result by induction on pairs $(N,{\rm rk}(\Sigma))$, with $N\geq {\rm rk}(\Sigma)\geq 1$.

\medskip

\noindent{\bf Base Cases.} If ${\rm rk}(\Sigma)=1$, then $s=1$, and therefore $I_a(\ell_1^{m_1})=\langle\ell_1^a\rangle$ (which has linear graded free resolution).

If $N={\rm rk}(\Sigma)$, then $s={\rm rk}(\Sigma)$ and $m_1=\cdots=m_s=1$. This is a particular case of item (f) in the Introduction, so this base case is also verified.

\medskip

\noindent{\bf Inductive Step.} Suppose $N>{\rm rk}(\Sigma)\geq 2$.

Let $\ell_1\in \Sigma$, and denote $I:=I_a(\ell_1^{m_1}\ell_2^{m_2}\cdots\ell_s^{m_s})$, and $J:=I_{a-1}(\ell_1^{m_1-1}\ell_2^{m_2}\cdots\ell_s^{m_s})$. First we want to show that $I:\ell_1=J$. Since $I=\ell_1J+I_a(\ell_1^{m_1-1}\ell_2^{m_2}\cdots\ell_s^{m_s})$, we obviously have $J\subseteq I:\ell_1$.

For the other inclusion, after a possible change of variables and embedding in a smaller ring, suppose ${\rm rk}(\Sigma)=n+1$.

Let $\Sigma':=(\underbrace{\ell_1,\ldots,\ell_1}_{m_1-1},\underbrace{\ell_2,\ldots,\ell_2}_{m_2},\ldots,\underbrace{\ell_s,\ldots,\ell_s}_{m_s})$. Obviously, $\Sigma'$ has generic support.

\medskip

First suppose that ${\rm rk}(\Sigma')=n$ (which necessarily implies that $m_1=1$). So, after a change of variables we can suppose $\ell_1=x_0$, and $\ell_2,\ldots,\ell_s\in R':=\mathbb K[x_1,\ldots,x_n]$.

Let $f\in I_a(x_0\ell_2^{m_2}\cdots\ell_s^{m_s}):x_0$. Then $x_0f=x_0g+h$, where $g\in I_{a-1}(\ell_2^{m_2}\cdots\ell_s^{m_s})$ and $h\in I_{a}(\ell_2^{m_2}\cdots\ell_s^{m_s})$. So $x_0(f-g)\in I_{a}(\ell_2^{m_2}\cdots\ell_s^{m_s})$, which, since $x_0$ is a nonzero divisor mod $I_{a}(\ell_2^{m_2}\cdots\ell_s^{m_s})$, leads to $f-g\in I_{a}(\ell_2^{m_2}\cdots\ell_s^{m_s})$. But this ideal is trivially included in $I_{a-1}(\ell_2^{m_2}\cdots\ell_s^{m_s})$, and therefore $f\in I_{a-1}(\ell_2^{m_2}\cdots\ell_s^{m_s})$.

\medskip

Suppose ${\rm rk}(\Sigma')=n+1$. As a reminder, $M:=\langle x_0,\ldots,x_n\rangle$.

We have $I\subseteq I^{\rm sat}\cap M^a$, leading to $$I:\ell_1\subseteq (I^{\rm sat}:\ell_1)\cap \underbrace{(M^a:\ell_1)}_{M^{a-1}}.$$ By inductive hypotheses, since ${\rm Supp}(\Sigma')$ remains generic, $J$ has linear graded free resolution, and so, from Remark \ref{remark0}, $J=J^{\rm sat}\cap M^{a-1}$. So in order to prove our desired inclusion, it is enough to show that $I^{\rm sat}:\ell_1\subseteq J^{\rm sat}$.

By Proposition \ref{prop_saturation} we have $$I^{\rm sat}=\bigcap_{\frak p\in \Gamma(\Sigma)\setminus\{M\}}\frak p^{a-N+\nu_{\Sigma}(\frak p)},$$ where $\Gamma(\Sigma)$ is the set of all linear primes that contain $I$, and that are generated by subsets of ${\rm Supp}(\Sigma)$; and

$$J^{\rm sat}=\bigcap_{\frak q\in \Gamma(\Sigma')\setminus\{M\}}\frak q^{(a-1)-(N-1)+\nu_{\Sigma'}(\frak q)},$$ where $\Gamma(\Sigma')$ is the set of all linear primes that contain $J$, and that are generated by subsets of ${\rm Supp}(\Sigma')$.

Since $I\subset J$ (because $I_a(\ell_1^{m_1-1}\ell_2^{m_2}\cdots\ell_s^{m_s})\subset J$), and ${\rm Supp}(\Sigma')\subseteq {\rm Supp}(\Sigma)$, we have $\Gamma(\Sigma')\subseteq \Gamma(\Sigma)$, and hence
$$I^{\rm sat}:\ell_1\subseteq\bigcap_{\frak q\in \Gamma(\Sigma')\setminus\{M\}}(\frak q^{a-N+\nu_{\Sigma}(\frak q)}:\ell_1).$$

Let $\frak q\in\Gamma(\Sigma')\setminus\{M\}$.
\begin{itemize}
  \item[(1)] If $\ell_1\notin\frak q$, then $\nu_{\Sigma'}(\frak q)=\nu_{\Sigma}(\frak q)$. So $$\frak q^{a-N+\nu_{\Sigma}(\frak q)}:\ell_1= \frak q^{a-N+\nu_{\Sigma}(\frak q)}=\frak q^{(a-1)-(N-1)+\nu_{\Sigma'}(\frak q)}.$$
  \item[(2)] If $\ell_1\in\frak q$, then $\nu_{\Sigma'}(\frak q)=\nu_{\Sigma}(\frak q)-1$, so
$$\frak q^{a-N+\nu_{\Sigma}(\frak q)}:\ell_1=\frak q^{a-N+\nu_{\Sigma}(\frak q)-1}=\frak q^{(a-1)-(N-1)+\nu_{\Sigma'}(\frak q)}.$$
\end{itemize}

Therefore, from items (1) and (2) above, and by Proposition \ref{prop_saturation}, we obtained the desired inclusion $I^{\rm sat}:\ell_1\subseteq J^{\rm sat}$, and therefore $$I:\ell_1=J.$$

\smallskip

Now this equality gives the short exact sequence of $R$-graded modules:

$$0\longrightarrow R(-1)/J\stackrel{\cdot\ell_1}\longrightarrow R/I\longrightarrow R/\langle\ell_1,I\rangle\longrightarrow 0.$$

\smallskip

Since $J$ is generated in degree $a-1$, by inductive hypotheses and \cite[Theorem 1.2]{EiGo}, the Castelnuovo-Mumford regularity of the leftmost nonzero module is $${\rm reg}\left(R(-1)/J\right)=(a-2)+1=a-1.$$

Considering the rightmost nonzero module $R/\langle\ell_1,I\rangle$, after a change of variables, we may suppose $\ell_1=x_0$. For $j=2,\ldots,s$, let $\ell_j=c_jx_0+\bar{\ell}_j$, where $c_j\in\mathbb K$ and $\bar{\ell}_j\in S:=\mathbb K[x_1,\ldots,x_n]$. Let $\overline{\Sigma}=(\underbrace{\bar{\ell}_2,\ldots,\bar{\ell}_2}_{m_2},\ldots,\underbrace{\bar{\ell}_s,\ldots,\bar{\ell}_s}_{m_s})$. Then $|\overline{\Sigma}|=m_2+\cdots+m_s\leq N-1$, and ${\rm rk}(\overline{\Sigma})=n$. Moreover, one has $$\langle\ell_1, I\rangle=\langle x_0, I_a(\overline{\Sigma})\rangle$$
and therefore $R/\langle\ell_1,I\rangle\cong S/I_a(\overline{\Sigma})$.

Suppose there exist $d_1,\ldots,d_n\in\mathbb K$ such that
$d_1\bar{\ell}_{i_1}+\cdots+d_n\bar{\ell}_{i_n}=0$, for some $2\leq i_1<\cdots<i_n\leq s$. Then we have $d_1\ell_{i_1}+\cdots+d_n\ell_{i_n}-(c_{i_1}d_1+\cdots+c_{i_n}d_n)\ell_1=0$. But these $n+1$ linear forms are linearly independent because ${\rm Supp}(\Sigma)$ is generic, so $d_1=\cdots=d_n=0$. This leads to ${\rm Supp}(\overline{\Sigma})$ being generic as well.

If $a\leq |\overline{\Sigma}|$, then by induction hypotheses, ${\rm reg}(S/I_a(\overline{\Sigma}))=a-1$, and so via \cite[Corollary 4.6]{Ei}, ${\rm reg}(R/\langle\ell_1,I\rangle)=a-1$. On the other hand, if $a> |\overline{\Sigma}|$, then  $I_a(\overline{\Sigma})=0$ by convention, so ${\rm reg}(R/\langle\ell_1,I\rangle)={\rm reg}(R/\langle\ell_1\rangle)=0$.

\medskip

Finally, \cite[Corollary 20.19 b.]{Ei2} gives that ${\rm reg}(R/I)\leq a-1$, and since $I$ is generated in degree $a$, we obtain that $I$ has linear graded free resolution.
\end{proof}

With the notations and discussions at the beginning of this section, we have the following result.

\begin{cor}\label{cor_decomposition} Let $\Sigma:=(\underbrace{\ell_1,\ldots,\ell_1}_{m_1},\ldots,\underbrace{\ell_s,\ldots,\ell_s}_{m_s})$ be a collection of linear forms in $R:=\mathbb K[x_0,\ldots,x_n]$, with $s, m_1,\ldots,m_s\geq 1$, ${\rm rk}(\Sigma)=n+1$, and with ${\rm Supp}(\Sigma)$ generic. Set $N:=m_1+\cdots+m_s$. Then for any $1\leq a\leq N$, we have the decomposition
$$I_a(\ell_1^{m_1}\ell_2^{m_2}\cdots\ell_s^{m_s})=\bigcap_{\frak p\in \Gamma(\Sigma)}\frak p^{a-N+\nu_{\Sigma}(\frak p)} =\bigcap_{c=1}^s\left(\bigcap_{1\leq i_1<\cdots<i_c\leq s}\langle\ell_{i_1},\ldots,\ell_{i_c}\rangle^{\mu(i_1,\ldots,i_c)}\right).$$
\end{cor}

\begin{proof} By Theorem \ref{thm_linear}, the ideal $I:=I_a(\ell_1^{m_1}\ell_2^{m_2}\cdots\ell_s^{m_s})$ has linear graded free resolution. Therefore, by Remark \ref{remark0}, $I=I^{\rm sat}\cap M^a$. Since $\nu_{\Sigma}(M)=N$, by Proposition \ref{prop_saturation}, we get the first equality in the statement. The second equality has been addressed in the discussions right before presenting Proposition \ref{prop_saturation}.
\end{proof}

It is worth noting that since the first equality presents a decomposition of $I$ that is irredundant and irreducible, we have $\Gamma(\Sigma)={\rm Ass}(I)$.

\section{Star configurations}

\subsection{The Geramita-Harbourne-Migliore Conjecture.} Let $\A=\{H_1,\ldots,H_s\}$ be a collection of $s\geq n+1$ hyperplanes in $\mathbb P^n$, and suppose $\ell_1,\ldots,\ell_s\in R$ are defining linear forms of these hyperplanes: i.e., $H_i=V(\ell_i), i=1,\ldots,s$; we will abuse notation by saying $\A=\{\ell_1,\ldots,\ell_s\}\subset R$. Suppose $\A$ is generic, and let $1\leq c\leq n$.

\begin{lem}\label{lemma1} The defining ideal of the star configuration $V_c(\A)$ satisfies $$I(V_c(\A))=I_{s-c+1}(\ell_1\cdots\ell_s).$$ Furthermore, for any $m\geq 1$, we have
$$I(V_c(\A))^m=I_{m(s-c+1)}(\ell_1^m\cdots\ell_s^m).$$
\end{lem}

\begin{proof} We prove the first part by induction on $|\A|=s\geq n+1$.
The base case of induction is $s=n+1$. In this instance, after a change of variables we can suppose that $\A$ consists of the coordinate hyperplanes, i.e., $\ell_i=x_{i-1}, i=1,\ldots,n+1$. So $$I(V_c(\A))=\bigcap_{0\leq j_1<\cdots<j_c\leq n}\langle x_{j_1},\ldots,x_{j_c}\rangle,$$ which by standard results on squarefree monomial ideals equals to $$I_{n-c+2}(x_0\cdots x_{n})=\langle \{x_{i_1}\cdots x_{i_{n-c+2}}\,|\,0\leq i_1<\cdots<i_{n-c+2}\leq n\}\rangle.$$

Suppose $s>n+1$. In the proof of \cite[Proposition 2.9 (4)]{GeHaMi} we have $$I(V_c(\A))=\ell_s\cdot I(V_c(\A'))+I(V_{c-1}(\A')),$$ where $\A'=\A\setminus\{V(\ell_s)\}$, and by inductive hypotheses we have

$$I(V_c(\A'))=I_{s-c}(\ell_1\cdots\ell_{s-1}) \mbox{ and }I(V_{c-1}(\A'))=I_{s-c+1}(\ell_1\cdots\ell_{s-1}).$$ But it is obvious that at the level of ideals generated by fold products of linear forms we have $$I_a(\ell_1\cdots \ell_s)=\ell_s\cdot I_{a-1}(\ell_1\cdots \ell_{s-1})+I_a(\ell_1\cdots \ell_{s-1}).$$ So the conclusion follows.

\medskip

The second part follows immediately from the simple observation that $(I_a(\ell_1\cdots \ell_s))^m=I_{ma}(\ell_1^m\cdots \ell_s^m)$.
\end{proof}

\smallskip

\begin{thm}\label{thm_GHMconj} Conjecture \ref{Conjecture1} is true in its full generality.
\end{thm}
\begin{proof} As we mentioned already in the Introduction, from Lemma \ref{lemma1} and Theorem \ref{thm_linear}, $I(V_c(\A))^m$ have linear graded free resolution, and therefore via Remark \ref{remark0} and from \cite[Corollary 4.9]{GeHaMi}, Conjecture \ref{Conjecture1} is true.
\end{proof}

\smallskip

\begin{rem}\label{rem_GHM} As an exercise, we will show that the decomposition of $(I_{m(s-c+1)}(\ell_1^m\cdots\ell_s^m))^{\rm sat}$ presented in Proposition \ref{prop_saturation} matches the one presented in \cite[Corollary 4.9]{GeHaMi}.

Let $\Sigma=(\underbrace{\ell_1,\ldots,\ell_1}_{m},\ldots,\underbrace{\ell_s,\ldots,\ell_s}_{m})$. So in Proposition \ref{prop_saturation} we have $N=ms$, and $a=m(s-c+1)$.

As we have seen, any associated prime of $I$ (so an element in $\Gamma(\Sigma)$) is generated by at least $N-a+1=m(c-1)+1$ linear forms from $\Sigma$. So we must pick at least $c$ linear forms from $\A:=\{\ell_1,\ldots,\ell_s\}$ to be able to generate such a minimal prime (indeed picking $c-1$ elements from $\A$ each $m$ times it forces us to pick another linear form from the remaining $s-c$ elements of $\A$). Since any $n+1$ elements of $\A$ are linearly independent, since we don't want to obtain $M$, the most number of linear forms from $\A$ we can pick is $n$.

So, any codimension $j$ associated prime $\frak p$ over $I$, where $c\leq j\leq n$, is of the form $\frak p=\langle \ell_{i_1},\ldots,\ell_{i_j}\rangle$, where $1\leq i_1<\cdots<i_j\leq s$. \footnote{We can adjust conveniently how many times we pick each $\ell_{i_1},\ldots,\ell_{i_j}$, such that we have $m(c-1)+1$ linear forms chosen from $\Sigma$.} But in this situation, $\nu_{\Sigma}(\frak p)=mj$.

With $a-N+\nu_{\Sigma}(\frak p)=m(j-c+1)$, we get that the codimension $j$ component of $I^{\rm sat}$ is $$\bigcap_{1\leq i_1<\cdots<i_j\leq s}\langle\ell_{i_1},\ldots,\ell_{i_j}\rangle^{m(j-c+1)},$$ which equals $I(V_j(\A))^{(m(j-c+1))}$. As $j=c,\ldots,n$, we obtain indeed the decomposition in \cite[Corollary 4.9]{GeHaMi}.
\end{rem}

We end this subsection with an observation on how to prove Conjecture \ref{Conjecture1} in a new case, without appealing to the technique we have developed in the previous section.

\begin{rem}\label{remark2} With the identification in Lemma \ref{lemma1} established, the main result in \cite{BrCoVa} will give that Conjecture \ref{Conjecture1} is true for $c=2$ and also for the already proven case $c=n$. Here is how:

\medskip

\noindent\underline{$c=2$}. At the end of the proof of \cite[Proposition 2.1]{To2}, we can see that for any $\Sigma=(L_1,\ldots,L_N)$ and any $1\leq a\leq N$, $I_a(L_1\cdots L_N)$ is generated by the maximal minors of an $a\times N$ matrix with linear entries. The matrix looks like

$$X:=\left[\begin{array}{cccc}c_{11}L_1&c_{12}L_2&\cdots&c_{1N}L_N\\c_{21}L_1&c_{22}L_2&\cdots&c_{2N}L_N\\ \vdots & \vdots& &\vdots\\ c_{a1}L_1&c_{a2}L_2&\cdots&c_{aN}L_N\end{array}\right],$$ where the constants $c_{uv}$ are generic (meaning that none of the $a\times a$ minors of $X$ vanish. In fact, for any $j=1,\ldots,a$, the ideal generated by the $j\times j$ minors of $X$, denoted $I_j(X)$, equals $I_j(L_1\cdots L_N)$.

Let $c=2$,  $N=s$, $a=s-c+1=s-1$, and $\Sigma=(\ell_1,\ldots,\ell_s)$, where $\ell_i$ define the generic arrangement~$\A$.
By Lemma \ref{lemma1}, for $1\leq j\leq a=s-1$, we have
$$I_j(X)=I_j(\ell_1\cdots\ell_s)=I(V_{s-j+1}(\A)),$$ and therefore

\begin{itemize}
  \item ${\rm ht}(I_1(X))=n+1$,
  \item ${\rm ht}(I_a(X))=2$,
  \item ${\rm ht}(I_j(X))=\min\{s-j+1,n+1\}$, for $j=2,\ldots,a-1=s-2$.
\end{itemize}
Since $$2=s-(s-1)+1=s-a+1 \mbox{ and } s-j+1=(s-1+1-j)(s-(s-1))+1=(a+1-j)(s-a)+1,$$
  we satisfy the conditions of \cite[Theorem 3.7]{BrCoVa}, and therefore, any power of $I_a(X)=I(V_2(\A))$ has linear graded free resolution, hence by Remark \ref{remark0}, the conjecture is true for $c=2$.

\medskip

\noindent\underline{$c=n$}. From Lemma \ref{lemma1} we have $I(V_n(\A))=I_{s-n+1}(\ell_1\cdots\ell_s)$, which also equals $I_{s-n+1}(X)$. So ${\rm ht}(I_{s-n+1}(X))=n$.

Let $1\leq j\leq s-n$. We still have $I_j(X)=I_j(\ell_1\cdots\ell_s)$. Because $\A$ is generic, the linear code $\mathcal C_{\A}$ has minimum distance $d_1(\mathcal C_{\A})=s-(n+1)+1=s-n$. By \cite[Proposition 2.2]{AnGaTo}, for all $1\leq j\leq s-n$, we have ${\rm ht}(I_j(\ell_1\cdots\ell_s))=n+1$. Therefore  the conditions in \cite[Theorem 3.7]{BrCoVa} are satisfied and the conjecture is also true for $c=n$.
\end{rem}

\subsection{Resurgence.} We finish the paper by determining the resurgence of defining ideals of star configurations.

Let $0\neq I \subset R:=\mathbb K[x_0,\ldots,x_n]$ be a homogeneous ideal. The {\it resurgence} of $I$ is defined as
$$
\rho(I):={\rm sup}\left\{\frac{m}{r} \,\big|\, I^{(m)}\nsubseteq I^r\right\}.
$$
The resurgence of ideals is an important invariant describing the containment of symbolic powers and regular powers. We always have $\rho(I)\geq 1$ since if $m<r$, then $I^{(m)}\nsubseteq I^r$. By a result of Ein-Lazarsfeld-Smith \cite{ELS01} and Hochster-Huneke \cite{HH02} that $I^{(Nr)}\subseteq I^r$ for any $r\geq 1$, one can see that $\rho(I)\leq N$.

\medskip

Let $\A=\{H_1, \ldots, H_s\}$ be a set of generic $s\geq n+1$ hyperplanes in $\mathbb{P}^n$ with defining linear forms $\ell_1, \ldots, \ell_s$, i.e., $H_i=V(\ell_i)$ for $1\leq i\leq s$.  Let $N=s-1$ and $\A^{\prime}=\{H^{\prime}_0, H^{\prime}_1, \ldots, H^{\prime}_N\}$ be the set of coordinate hyperplanes in $\mathbb{P}^N$, i.e., $H^{\prime}_i=V(z_i)$, where $z_0, \ldots z_N$ are coordinate variables.

Define $$\phi: \mathbb{K}[z_0, \ldots z_N] \longrightarrow \mathbb{K}[x_0, \ldots, x_n]$$ by $\phi: z_i \longmapsto \ell_{i+1}$ for $0\leq i\leq N$.

Set $I:=I\left(V_c(\A, \mathbb{P}^n)\right)$ (resp. $I^{\prime}:=I\left(V_c(\A^{\prime}, \mathbb{P}^N)\right)$) be the defining ideal of the star configuration of codimension $c$ with support in $\A$ (resp. in $\A^{\prime}$). By \cite[Theorem 3.1]{GeHaMi}, one has that $\phi({I^{\prime}}^{(m)})=I^{(m)}$ for every $m\geq 1$.

In the following, we  first give a positive answer to Question 4.12 in \cite{GeHaMi}, that asks whether $\rho(I)=\rho(I^{\prime})$. Then we prove a result that calculates the resurgence of the defining ideal of any star configuration, which generalizes \cite[Theorem 4.11]{GeHaMi}.

\begin{prop} \label{prop_question} With the above setting, one has $\rho(I)=\rho(I^{\prime})$ for  $1\leq c\leq n$.
\end{prop}

\begin{proof} By the discussion after \cite[Question 4.12]{GeHaMi}, one has that $\rho\left(I\right)\leq \rho\left(I^{\prime}\right)$, so we only need to prove $\rho\left(I\right)\geq \rho\left(I^{\prime}\right)$.

Assume $I^{(m)}\subseteq I^r$. We want to show ${I^{\prime}}^{(m)}\subseteq {I^{\prime}}^r$, which yields the desired inequality.

By \cite[Theorem 3.1]{GeHaMi}, the symbolic power ${I^{\prime}}^{(m)}$ is generated by monomials in $z_i$. Let $\xi^{\prime}\in {I^{\prime}}^{(m)}$ be a monomial generator, then one can write $\xi^{\prime}=z_0^{t_1}z_1^{t_2} \cdots z_N^{t_s}$. Hence $\xi:=\phi(\xi^{\prime})=\ell_1^{t_1}\ell_2^{t_2}\cdots \ell_s^{t_s}\in \phi({I^{\prime}}^{(m)})=I^{(m)}$. Since $I^{(m)}\subseteq I^r=I_{ra}(\ell_1^r\cdots\ell_s^r)$, where $a=s-c+1$ (see Lemma \ref{lemma1}). Hence $$\xi=\ell_1^{t_1}\ell_2^{t_2}\cdots \ell_s^{t_s}=\left(\ell_1^{k_1}\ell_2^{k_2}\cdots \ell_s^{k_s}\right)\left(\ell_1^{t_1-k_1}\ell_2^{t_2-k_2}\cdots \ell_s^{t_s-k_s}\right),$$ where $k_j\leq r, t_j-k_j\geq 0$ for $1\leq j\leq s$, and $\sum_{j=1}^sk_j=ra$.

In the same manner, $$\xi^{\prime}=z_0^{t_1}z_1^{t_2}\cdots z_N^{t_s}=\left(z_0^{k_1}z_1^{k_2}\cdots z_N^{k_s}\right)\left(z_0^{t_1-k_1}z_1^{t_2-k_2}\cdots z_N^{t_s-k_s}\right).$$
Since $k_j\leq r$ for $1\leq j\leq s$ and $\sum_{j=1}^sk_j=ra$, one has $z_0^{k_1}z_1^{k_2}\cdots z_N^{k_s}\in I_{ra}(z_0^r z_1^r \cdots z_N^r)={I^{\prime}}^r$. Therefore $\xi^{\prime}\in {I^{\prime}}^r$.
\end{proof}

\smallskip

\begin{thm}\label{thm_resurgence} $I:=I\left(V_c(\A, \mathbb{P}^n)\right)$ be the defining ideal of the star configuration of codimension $c$ with support in $\A$. Then $$\rho(I)=\frac{c(s-c+1)}{s}.$$
\end{thm}
\begin{proof} By Proposition \ref{prop_question}, setting $s=N+1$, it is enough to show $\displaystyle \rho(I^{\prime})=\frac{c(N-c+2)}{N+1}$, where $I^{\prime}:=I\left(V_c(\A^{\prime}, \mathbb{P}^N)\right)$ is the defining ideal of the star configuration of codimension $c$ with support in $\A^{\prime}=\{V(z_0),\ldots,V(z_N)\}$, the coordinate hyperplane arrangement.

As mentioned after \cite[Definition 4.10]{GeHaMi}, citing \cite{BoHa}, we have that $\displaystyle \rho(I^{\prime})\geq\frac{c(N-c+2)}{N+1}$, so we have to show the other inequality; i.e., if $m/r\geq c(N-c+2)/(N+1)$, then ${I^{\prime}}^{(m)}\subseteq {I^{\prime}}^r$.

Let $\xi=z_0^{t_0}\cdots z_N^{t_N}\in {I^{\prime}}^{(m)}$ be an arbitrary monomial generator. Then for every $0\leq i_1<\cdots<i_c\leq N$,
$\xi\in\langle z_{i_1},\ldots,z_{i_c}\rangle^m.$ So for every $0\leq i_1<\cdots<i_c\leq N$, we have
$$t_{i_1}+\cdots+t_{i_c}\geq m \geq cr\left(1-\frac{c-1}{N+1}\right).$$

By Lemma \ref{lemma1} and Corollary \ref{cor_decomposition}, one has
$$
{I^{\prime}}^r=I_{r(N-c+2)}(z_0^r\cdots z_N^r)=\bigcap_{c^{\prime}=c}^{N+1}\left(\bigcap_{0\leq i_1<\cdots<i_{c^{\prime}}\leq N}\langle z_{i_1},\ldots, z_{i_{c^{\prime}}}\rangle^{r(c^{\prime}-c+1)}\right),
$$ so we need to show $\xi\in \langle z_{i_1},\ldots, z_{i_{c^{\prime}}}\rangle^{r(c^{\prime}-c+1)}$, for every $0\leq i_1<\cdots<i_{c'}\leq N$, where $c\leq c'\leq N+1$.

Let $0\leq i_1<\cdots<i_{c'}\leq N$, with $c\leq c'\leq N+1$. From above, for any $U\subset\{i_1,\ldots,i_{c'}\}$ with $|U|=c$ we have $$\sum_{u\in U}t_u\geq cr\left(1-\frac{c-1}{N+1}\right).$$ Summing up all of these inequalities over all such subsets $U$ we have

$${{c'-1}\choose{c-1}}(t_{i_1}+\cdots+t_{i_{c'}})\geq {{c'}\choose{c}}cr\left(1-\frac{c-1}{N+1}\right),$$ leading to
$$t_{i_1}+\cdots+t_{i_{c'}}\geq c'r\left(1-\frac{c-1}{N+1}\right).$$ But the last quantity is $\geq r(c'-c+1)$, since $c'\leq N+1$, concluding the proof.
\end{proof}

\bigskip

\noindent
{\bf Acknowledgment.} \, The authors would like to thank Kuei-Nuan Lin and Yi-Huang Shen for careful reading of an earlier version of this manuscript where we claimed we are proving Conjecture \ref{Conjecture2} in its full generality. Their corrections made us restrict the hypotheses to collections of linear forms with generic support; consequently, that conjecture is still open in the general case.

\bigskip

\renewcommand{\baselinestretch}{1.0}
\small\normalsize 

\bibliographystyle{amsalpha}

\end{document}